\documentclass[runningheads]{lncse}

\usepackage{amsmath}
\usepackage{amsfonts}
\usepackage{epsfig}
\usepackage{graphics} 
\usepackage[ruled]{algorithm2e}
\usepackage{algorithmic}
\usepackage{color}
\usepackage{float}

\newcommand{\bx}{{\boldsymbol x}}
\newcommand{\bn}{{\boldsymbol n}}
\newcommand{\mB}{{\mathcal B}}
\newcommand{\Uad}{{U_{\text{ad}}}}

\begin{document}

\title{POD for optimal control of the Cahn-Hilliard system using spatially adapted snapshots}

\titlerunning{POD for control of Cahn-Hilliard using spatially adapted snapshots}

\author{ Carmen Gr\"a\ss{}le, Michael Hinze \and Nicolas Scharmacher  }

\authorrunning{C. Gr\"a\ss{}le, M. Hinze, N. Scharmacher}   

\institute{
Universit\"at Hamburg, Bundesstra\ss{}e 55, 20146 Hamburg, Germany, {\tt \{carmen.graessle, michael.hinze, nicolas.scharmacher\}@uni-hamburg.de}
}

\maketitle

\begin{abstract}
The present work considers the optimal control of a convective Cahn-Hilliard system, where the control enters through the velocity in the transport term. We prove the existence of a solution to the considered optimal control problem. For an efficient numerical solution, the expensive high-dimensional PDE systems are replaced by reduced-order models utilizing proper orthogonal decomposition (POD-ROM). The POD modes are computed from snapshots which are solutions of the governing equations which are discretized utilizing adaptive finite elements. The numerical tests show that the use of POD-ROM combined with spatially adapted snapshots leads to large speedup factors compared with a high-fidelity finite element optimization.
\end{abstract}

\section{Introduction}
The optimal control of two-phase systems has been studied in various papers, see e.g. \cite{HRV10}, \cite{HW12} and \cite{RS15}. In this paper, we concentrate our investigations on the diffuse interface 
approach, where we assume the existence of interfacial regions of small width between the phases. This has the advantage that topology changes like droplet collision or coalescence can be handled in a natural way.
Many degrees of freedom are needed in the interfacial regions in order to resemble the steep gradients well, whereas in the pure phases a small number of degrees of freedom is sufficient. Thus, in order to make numerical computations feasible, we utilize adaptive finite element methods. However, the optimization of a phase field model is still a costly issue, since a sequence of large-scale systems has to be solved repeatedly. For this reason, we replace the high-dimensional systems by low-dimensional POD approximations. This has been done in e.g. \cite{Vol01} for uniformly discretized snapshots.\\
We perform POD based optimal control using spatially adapted snapshots. The combination of POD with adaptive 
finite elements has been investigated for time-dependent problems in \cite{URL16} and \cite{GH17}.

\section{Convective Cahn-Hilliard system}
We consider the Cahn-Hilliard system which was introduced in \cite{CH58} as a model for phase transitions in binary systems. In a bounded and open domain \mbox{$\Omega \subset \mathbb{R}^{\mathtt{d}}$}, ${\mathtt{d} \in \{2,3\}}$, with Lipschitz boundary $\partial \Omega$, we assume two substances $A$ and $B$ to be given. In order to describe the spatial distribution over time $I=(0,T]$ with fixed end time $T>0$, a phase field variable $\varphi$ is introduced which fulfills $\varphi(t,\bx) = +1$ in the pure $A$-phase and $\varphi(t,\bx) = -1$ in the pure $B$-phase. Values of $\varphi$ between $-1$ and $+1$ represent the interfacial area between the two substances. Introducing the chemical potential $\mu$, the Cahn-Hilliard equations can be written as a coupled system of second-order in space
 \begin{equation}\label{CHcoupled_nocontrol} 
   \left\{
 \begin{array}{rcll}
 \varphi_t  + \mathtt{v} \cdot \nabla \varphi - b\Delta \mu & = &  0 &  \quad \text{in } I \times \Omega,\\
  -\sigma \varepsilon \Delta \varphi + \frac{\sigma}{\varepsilon} \mathcal{F}'(\varphi) & = & \mu  & \quad \text{in } I \times \Omega,\\  
   \partial_n \varphi = \partial_n \mu &  = &  0 & \quad  \text{on } I \times \partial \Omega,\\
  \varphi(0,\cdot) & = & \varphi_0 &   \quad \text{in } \Omega.
 \end{array}
 \right.
 \end{equation}

\noindent  By $b>0$ we denote a constant mobility, $\sigma > 0$ describes the surface tension and $\varepsilon >0$ is a parameter which is related to the interface width. 
  \noindent For the free energy $\mathcal{F}$, we consider the smooth polynomial free energy  (see e.g. \cite{EZ86}) \vspace{-0.1cm}
  \begin{equation*}
  \mathcal{F}(\varphi)  = \frac{1}{4}(1-\varphi^2)^2. 
  \end{equation*}
  

 \noindent A possible flow of the mixture at a given velocity field $\mathtt{v}$ is modeled in \eqref{CHcoupled_nocontrol} by the transport term which, in the context
 of multiphase flow, represents the coupling to the Navier-Stokes system, see e.g. \cite{HH77} and \cite{AGG12}. We use the following notations and assumptions:\\

 \noindent   \textit{Notations 2.1}\\[0.1cm]
 We denote by $H_{(0)}^1(\Omega)$ the space of functions in $H^1(\Omega)$ with zero mean value and by $L_\sigma^2(\Omega)^\mathtt{d} = \{f \in L^2(\Omega)^\mathtt{d}: \text{div}f = 0, f \cdot \bn_\Omega |_{\partial \Omega} = 0\}$ the space of solenoidal vector fields, for which we refer to \cite{Soh01} for details about well-definedness. We use as the solution space for the phase field variable $W(0,T) = \{f \in L^2(0,T;H_{(0)}^1(\Omega)) : f_t \in L^2(0,T;H_{(0)}^{-1}(\Omega))\}$.\\

 \renewcommand{\labelenumi}{\roman{enumi})}
 \noindent {\it Assumptions 2.2}
 \begin{enumerate}
  \item  The initial phase field $\varphi_0 \in H_{(0)}^1(\Omega)$ fulfills   
 $E_0 = E(\varphi_0)<\infty$ with Ginzburg-Landau free energy 
 \begin{equation*}
  E(\varphi) =  \int_\Omega  \frac{\sigma \varepsilon}{2} |\nabla \varphi|^2 + \frac{\sigma}{\varepsilon} \mathcal{F}(\varphi)d\bx.
 \end{equation*}
 \item The velocity $\mathtt{v}$ fulfills $ \mathtt{v} \in L^\infty(0,T;L_\sigma^2(\Omega)^\mathtt{d}) 
 \cap L^2(0,T;H^1(\Omega)^\mathtt{d})$.\\
 \end{enumerate}

  \noindent {\it Remark 2.3}\\[0.1cm]
   It is shown in \cite[Theorem 4.1.1]{Abe07} that
  there exists a unique solution 
  $(\varphi, \mu)$ to \eqref{CHcoupled_nocontrol}   with $\varphi \in 
  W(0,T)\cap L^2(0,T;H^2(\Omega))$, $\mu \in L^2(0,T;H^1(\Omega))$. This solution satisfies
  \begin{equation}\label{est-1}
    \|\varphi \|_{L^2(0,T;H^2(\Omega))}^2 + 
  \|\varphi_t\|_{L^2(0,T;H_{(0)}^{-1}(\Omega))}^2  \leq C \left( E_0 + \| \mathtt{v}\|_{L^2(0,T;L^2(\Omega)^\mathtt{d})}^2\right)
  \end{equation}
  where $C$ 
  is independent of $\mathtt{v}$ and $ \varphi_0$.

\section{Optimal control of Cahn-Hilliard}
  
\noindent We investigate the minimization of the quadratic objective functional
 \begin{equation*} 
  J(\varphi,u) = \frac{\beta_1}{2} \|\varphi - \varphi_d\|_{L^2(0,T;L^2(\Omega))}^2  + \frac{\beta_2}{2} \|\varphi-\varphi_T\|_{L^2(\Omega)}^2 + \frac{\gamma}{2} \|u\|_U^2
 \end{equation*}
   where $\beta_1,\beta_2 \geq 0$ are given constants, $\varphi_d \in L^2(0,T;L^2(\Omega))$ is the desired phase field, 
 $\varphi_T \in L^2(\Omega)$ is the target phase pattern at final time, $\gamma > 0$ is the penalty parameter and $u \in U = L^2(0,T;\mathbb{R}^m)$, with $m \in \mathbb{N}$, denotes the control variable which is a time-dependent variable and in particular independent of the current spatial discretization. The goal of the optimal control problem is to steer a given initial phase distribution $\varphi_0$ to a given desired phase pattern. This problem can also be interpreted as an optimal control of a free boundary which is encoded through the phase field variable. We consider distributed control which enters through the transport term:
 \begin{equation}\label{CHcoupled_withcontrol} 
   \left\{
 \begin{array}{rcll}
 \varphi_t  + (\mB u) \cdot \nabla \varphi- b\Delta \mu & = &  0 & \quad \text{ in } I \times \Omega,\\
  -\sigma \varepsilon \Delta \varphi +  \frac{\sigma}{\varepsilon} \mathcal{F}'(\varphi) & = &  \mu & \quad \text{ in } I \times \Omega.\\
   \partial_n \varphi = \partial_n \mu & = &  0 & \quad  \text{on } I \times \partial \Omega,\\
   \varphi(0,\cdot) &  = &  \varphi_0 &  \quad \text{in } \Omega.
 \end{array}
 \right.
 \end{equation}
 The control operator $\mB$ is defined by $\mB: U \to L^2(0,T;H^1(\Omega)^\mathtt{d}), (\mB u)(t) = 
   \sum_{i=1}^m u_i(t) \chi_i$ where $\chi_i \in L_\sigma^2(\Omega)^\mathtt{d} \cap H^1(\Omega)^\mathtt{d}, 1 \leq i 
   \leq m$, represent given control shape functions. The admissible set of controls is  
   \begin{equation*}
    \Uad = \{u \in U \; | \;  u_a(t) \leq u(t) \leq u_b(t) \text{ in } \mathbb{R}^m \text{ a.e. in } [0,T]\}
   \end{equation*}
   with $u_a, u_b \in L^\infty(0,T;\mathbb{R}^m), u_a(t) \leq u_b(t)$ almost everywhere in $[0,T]$. The inequalities between vectors are understood componentwise. Then, the optimal control problem can be expressed as
 \begin{equation}\label{ocp}
  \min_u \hat{J}(u) = J(\varphi(u),u) \quad \text{ s.t. } \quad (\varphi(u),u) \text{ satisfies } \eqref{CHcoupled_withcontrol}  \quad \text{ and } \quad u\in \Uad.
 \end{equation}

 \begin{theorem}[Existence of an optimal control]
 Problem \eqref{ocp} admits a solution $\bar{u} \in \Uad$.
 \end{theorem}
 \begin{proof} The infimum $\inf_{u \in \Uad} \hat{J}(u)$ exists due to $\hat{J} \geq 0$ and $\Uad \neq \emptyset$. Let $\{u_n\}_{n \in \mathbb{N}} \subset \Uad$ be a minimizing sequence and $\{\varphi_n\}_{n \in \mathbb{N}}$ the corresponding sequence of states $\varphi_n = \varphi(u_{n})$. Since $\Uad$ is closed, convex and bounded in $L^2(0,T;\mathbb{R}^m)\supset L^{\infty}(0,T;\mathbb{R}^m)$, we can extract a subsequence (denoted by the same name), which converges weakly to some $\bar{u}\in \Uad$. Weak convergence 
  $\mB u_n \rightharpoonup \mB \bar{u}$ in $L^2(0,T;H^1(\Omega)^{\mathtt{d}})$ follows from the linearity and boundedness of $\mB$. Due to the energy estimate \eqref{est-1} 
  there exists a constant $M>0$ such that for all $n\in\mathbb{N}$ we have
  $$\|\varphi_n \|_{L^2(0,T;H^2(\Omega))}^2 + \|\varphi_{n,t}\|_{L^2(0,T;H_{(0)}^{-1}(\Omega))}^2 \leq M.$$
  Since $W(0,T)\cap L^2(0,T;H^2(\Omega))$ is reflexive, there exists another subsequence (denoted by the same name) that converges weakly to some $\bar{\varphi}
  \in W(0,T)\cap L^2(0,T;H^2(\Omega))$. It remains to show, that the 
  pair $(\bar{\varphi},\bar{u})$ 
 is admissible, i.e. $\bar{\varphi}=\varphi(\bar{u})$. While passing to the 
 limit in the weak formulation is clear for the linear terms, the nonlinear 
 ones require further investigation. Since 
 $W(0,T)\cap L^2(0,T;H^2(\Omega))$ is compactly embedded in 
 $L^2(0,T;H^1(\Omega))$ (see \cite[Sect. 8, Corr. 4]{Sim86}), the sequence 
 $\{\varphi_n\}_{n \in \mathbb{N}}$ converges strongly to 
 $\bar{\varphi}$ in $L^2(0,T;H^1(\Omega))$. For the control term we have 
 for $v\in H^1(\Omega)$ the splitting
 \begin{multline*}
 \int_0^T ( \mB u_n \cdot\nabla \varphi_n - \mB \bar{u}\cdot\nabla 
 \bar{\varphi},v )_{L^2(\Omega)} \,dt =
 \\ \int_0^T ( \mB u_n \cdot \nabla
 (\varphi_n - \bar{\varphi}),v )_{L^2(\Omega)} \,dt + 
 \int_0^T ( \mB (u_n-\bar{u})\cdot\nabla \bar{\varphi},v )_{L^2(\Omega)} \,dt.
 \end{multline*}
 Due to $\nabla \bar{\varphi} \in L^2(0,T;H^1(\Omega)^{\mathtt{d}})$,
 the product $v\nabla \bar{\varphi} \in L^2(0,T;L^2(\Omega)^{\mathtt{d}})$ 
 gives rise to a continuous linear functional
 on $L^2(0,T;H^1(\Omega)^{\mathtt{d}})$. Hence, the right term vanishes for $n\rightarrow\infty$ by definition of weak convergence. For the left term we estimate
 \begin{multline*}
 \left|  \int_0^T ( \mB u_n \cdot \nabla (\varphi_n - 
 \bar{\varphi}),v )_{L^2(\Omega)} \,dt \right| \leq \\ \| \mB u_n 
 \|_{L^2(0,T;H^1(\Omega)^{\mathtt{d}})} \| \varphi_n - \bar{\varphi} \|_{L^2(0,T;H^1(\Omega))} \| v \|_{H^1(\Omega)},
 \end{multline*}
 which also vanishes for $n\rightarrow \infty$. For the nonlinearity $\mathcal{F}'$ we infer 
 from 
  \begin{align*}
  | \mathcal{F}' (\varphi) - \mathcal{F}'(\psi) | \leq C (\varphi ^2 + \psi ^2) \left| \varphi - \psi \right|
  \end{align*}
  for all $\varphi,\psi \in \mathbb{R}$ and some $C>0$ the estimate
 \begin{multline*}
 \left| \int_0^T ( \mathcal{F}'(\varphi_n)-\mathcal{F}'(\bar{\varphi}),v )_{L^2(\Omega)} 
 \, dt \right| \leq \\
 C (\| \varphi_n^2 \|_{L^2(0.T;L^2(\Omega))} + \| \bar{\varphi}^2 \|_{L^2(0,T;L^2(\Omega))}) \| \varphi_n - \bar{\varphi} \|_{L^2(0,T;H^1(\Omega))} \| v \|_{H^1(\Omega)},
 \end{multline*} 
 which gives the 
 desired convergence due to $L^{\infty}(0,T;H^1(\Omega))\subset L^4(0,T;L^4(\Omega))$. 
 Finally, the lower semi-continuity of $J$ yields
 $$ \hspace*{4cm} J(\bar{\varphi},\bar{u})=\inf_{u \in \Uad} \hat{J}(u). \hspace{4cm} \qed$$

 \end{proof}
 
 Problem \eqref{ocp} is a non-convex programming problem, so that different minima might exist. Numerical solution methods will converge in a local minimum which is close to the initial point. In order to compute a locally optimal solution to \eqref{ocp}, we consider the first-order necessary optimality condition given by the variational inequality
  \begin{equation}\label{firstordercond}
   \langle \hat{J}'(\bar{u}),u-\bar{u} \rangle_{U',U} \geq 0 \quad \forall u \in \Uad.
  \end{equation}
Following the 
  standard adjoint techniques, we derive that \eqref{firstordercond} is equivalent to 
  \begin{equation}\label{vi}
   \int_0^T \sum_{i=1}^m \left( \gamma \bar{u}_i(t) + \int_\Omega (\chi_i(\bx) \cdot 
   \nabla \varphi(t,\bx))\bar{p}(t,\bx) d\bx \right)(u_i(t) - \bar{u}_i(t)) dt \geq 0 
  \end{equation}
 for all $u \in \Uad$ where the function $\bar{p}$ is a solution to the adjoint equations
 \begin{equation}\label{adjoint}
 \left\{
 \begin{array}{r c l l}
  -p_t - \sigma \varepsilon \Delta q + \frac{\sigma}{\varepsilon} \mathcal{F}''(\bar{\varphi})q
  - \mathcal{B}u \cdot \nabla p & = &  -\beta_1(\bar{\varphi}-\varphi_d) & \; \text{in } I \times \Omega,\\      
  -q - b\Delta p  & = & 0, & \; \text{in } I \times \Omega,\\
  \partial_n p = \partial_n q  & = & 0, & \; \text{on } I \times \partial \Omega, \\
  p(T, \cdot) & = & -\beta_2(\bar{\varphi}(T,\cdot)-\varphi_T), & \; \text{in } \Omega.
 \end{array}
 \right.
 \end{equation}
 The variable $\bar{\varphi}$ in \eqref{adjoint} denotes the solution to \eqref{CHcoupled_withcontrol} associated with an optimal control $\bar{u}$.

  \section{POD-ROM using spatially adapted snapshots}
The optimal control problem \eqref{ocp} is discretized by adaptive finite elements and solved by a standard projected gradient method with an Armijo line search rule. In order to replace the resulting high-dimensional PDEs by low-dimensional approximations, we make use of POD-ROM, see e.g. \cite{HLBR12} or \cite{Vol13}. The nonlinearity is treated using DEIM, cf \cite{CS10}. In order to combine POD-ROM with spatially adapted snapshots, we follow the ideas in \cite{URL16} and \cite{GH17}.

 \section{Numerical results}
 We consider the unit square $\Omega = (0,1)\times(0,1)$, the end time $T=0.0125$ and utilize a uniform time grid with time step size $\Delta t = 2.5 \cdot 10^{-5}$. The mobility is $b = 2.5 \cdot 10^{-5}$, the surface tension is $\sigma = 25.98$ and the interface parameter is set to $\varepsilon = 0.02$. In the cost functional we use $\gamma = 0.0001$, $\beta_1 = 20$ and $\beta_2 = 20$. We use $m=1$ control shape function given by $\chi(x) = (\text{sin}(\pi x_1)\text{cos}(\pi x_2), -\text{sin}(\pi x_2)\text{cos}(\pi x_1))^T$. The desired state is shown in Figure \ref{fig:c_desired}. The initial state $\varphi_0$ coincides with $\varphi_d(0)$.  
 \noindent In order to fulfill the Courant-Friedrichs-Lewy (CFL) condition, we impose the control constraints $u_a = 0, u_b = 50$ and demand $h_{\min} > 0.00177$.\\
 The optimization is initialized with an input control $u=0 \in \Uad$. We compute the POD basis with respect to the $L^2(\Omega)$-inner product for the snapshot ensemble formed by the desired phase field $\varphi_d$, which is discretized using adaptive finite elements. Figure \ref{fig:c} shows the finite element solution and the POD solution for the phase field using $\ell=10$ and $\ell=20$ POD modes, respectively. It turns out that a large number of POD modes is needed in order to smoothen out oscillations due to the convection.
 
  \begin{figure}[H]
  \centering
  \includegraphics[scale=0.1113]{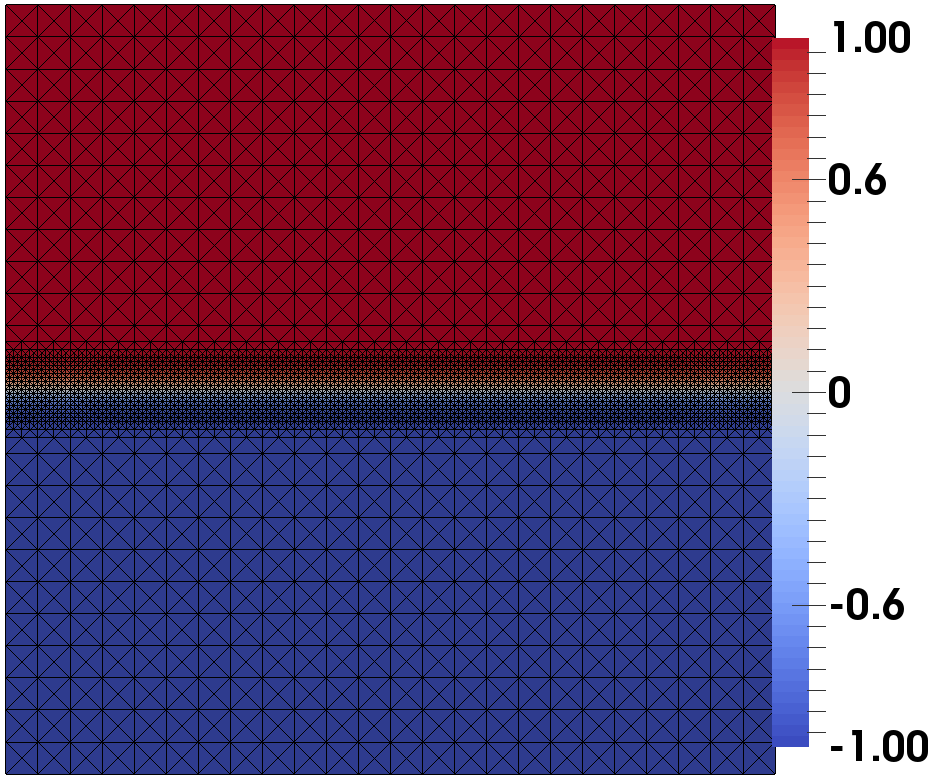} \hspace{-0.2cm}
   \includegraphics[scale=0.1113]{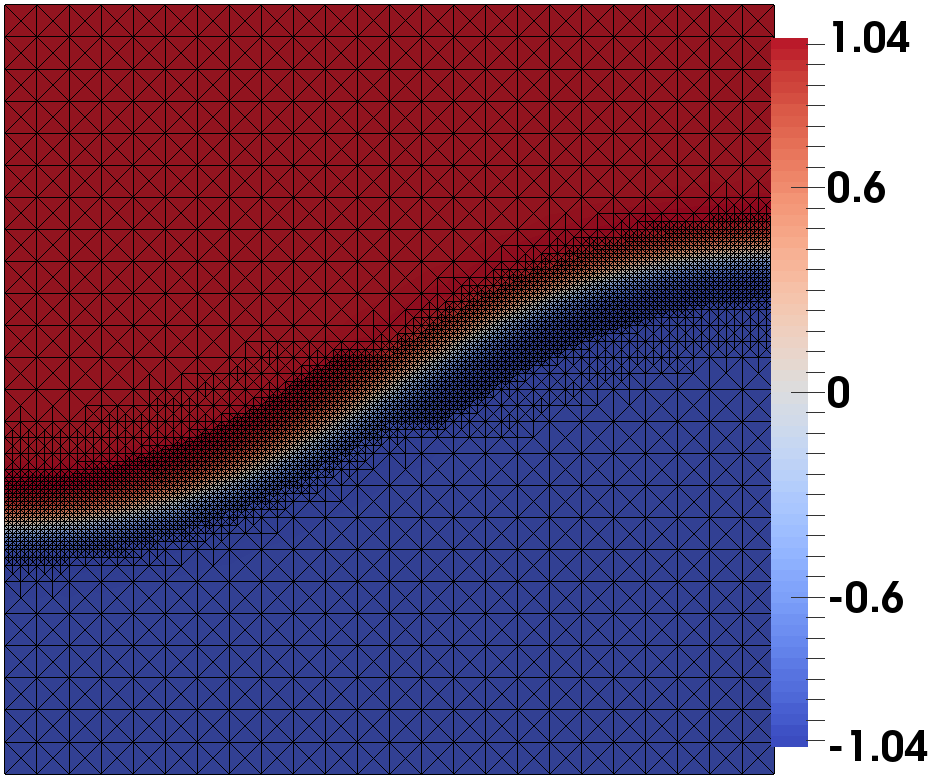} \hspace{-0.2cm}
   \includegraphics[scale=0.1113]{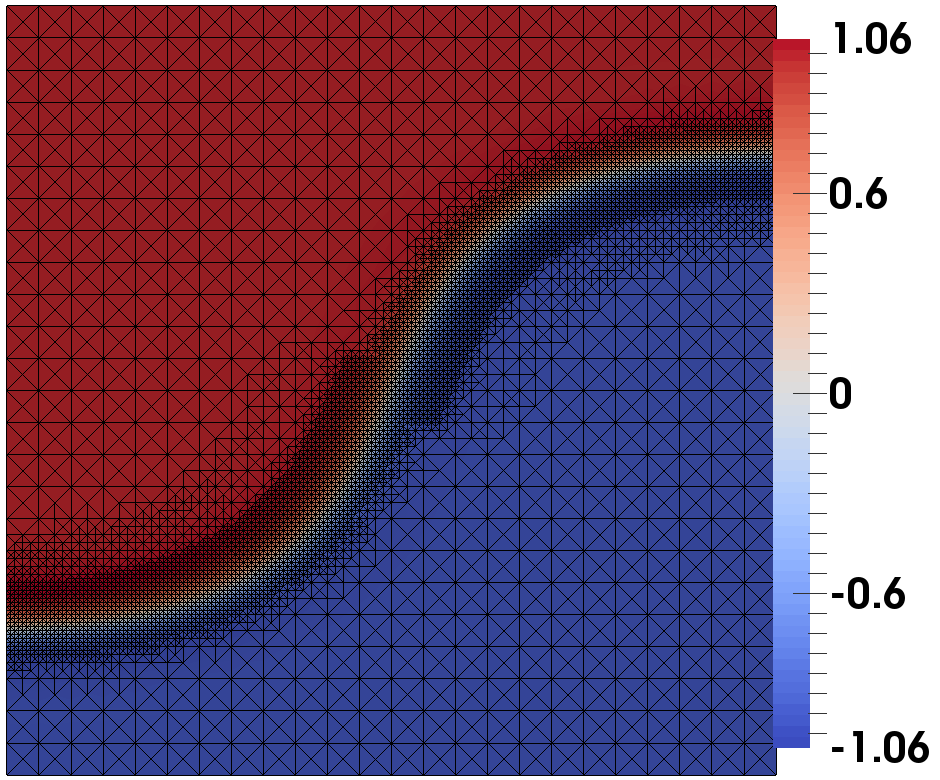}
    \caption{\em Desired phase field at $t_0, t_{250}, t_{500}$ with adaptive 
    meshes}
  \label{fig:c_desired}
  \end{figure}  \vspace{-1cm}
  \begin{figure}[H]
  \centering
  \includegraphics[scale=0.1113]{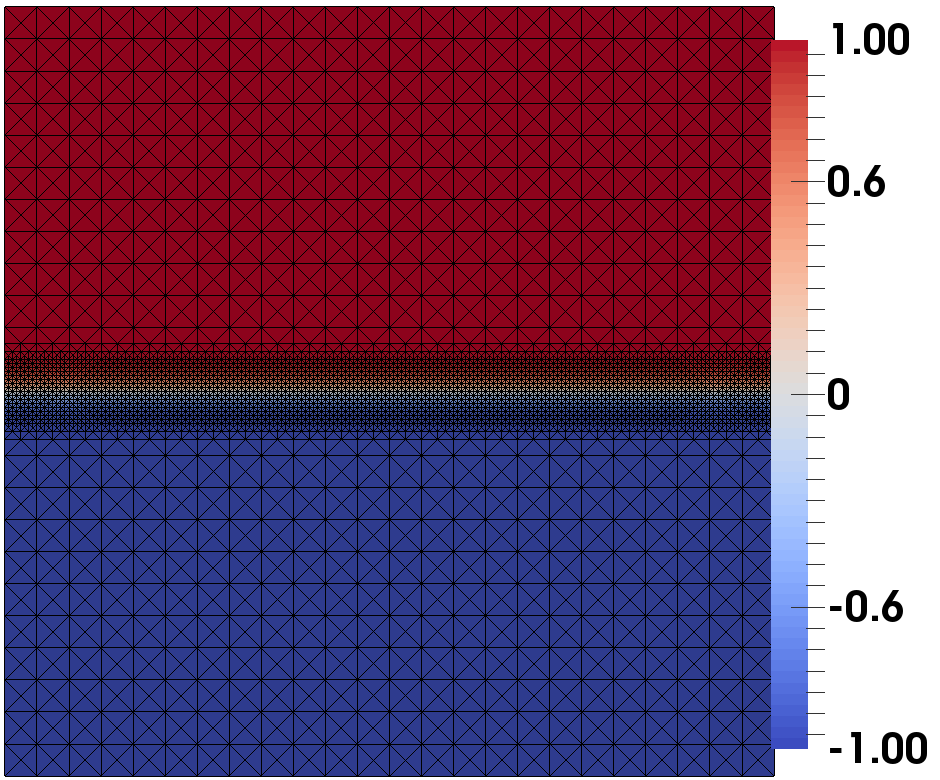} \hspace{-0.2cm}
   \includegraphics[scale=0.1113]{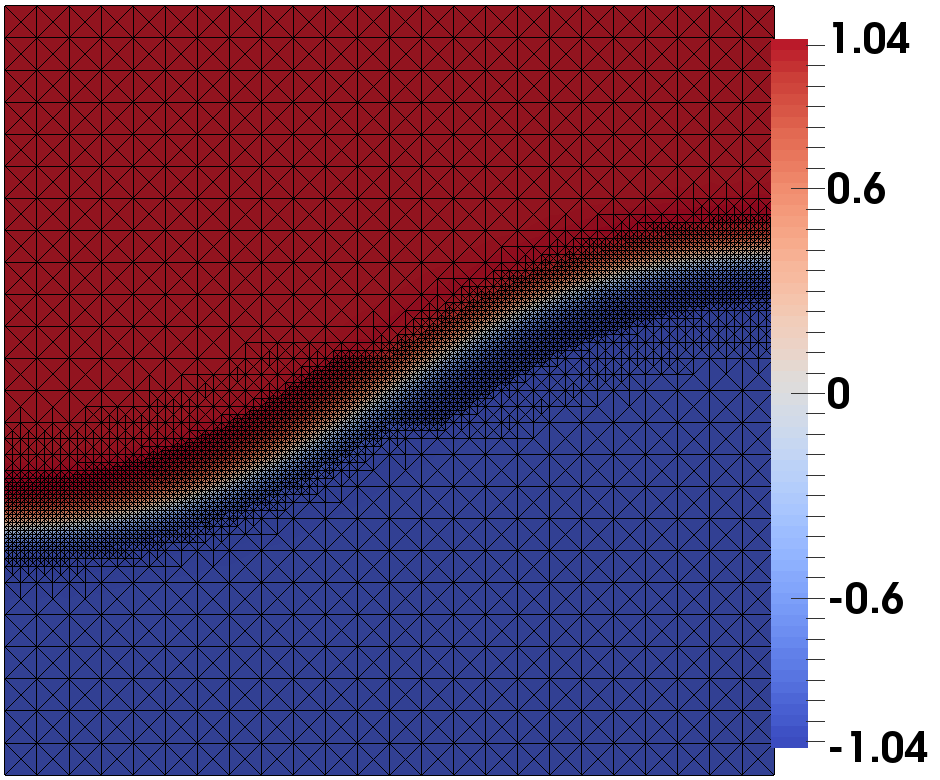} \hspace{-0.2cm}
   \includegraphics[scale=0.1113]{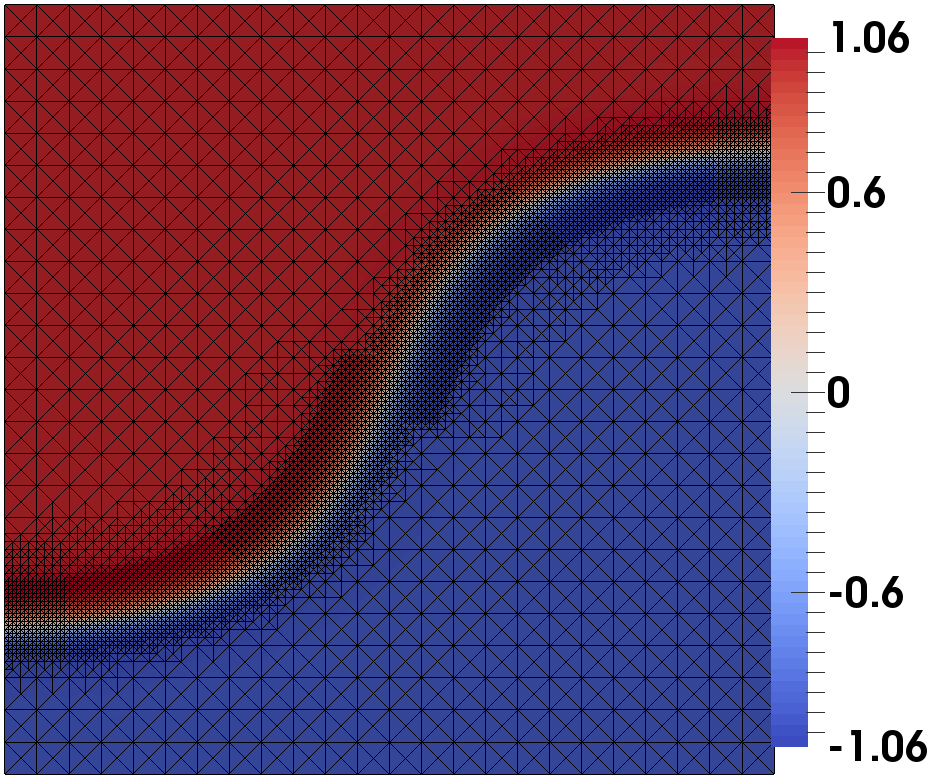}\\
   \includegraphics[scale=0.1113]{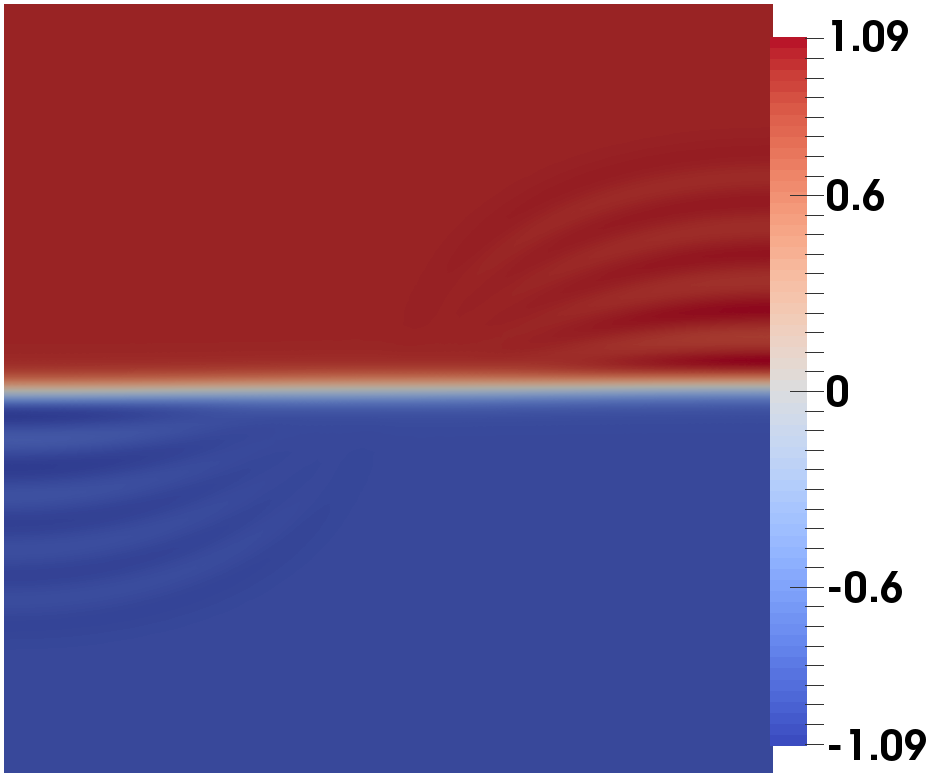} \hspace{-0.2cm}
    \includegraphics[scale=0.1113]{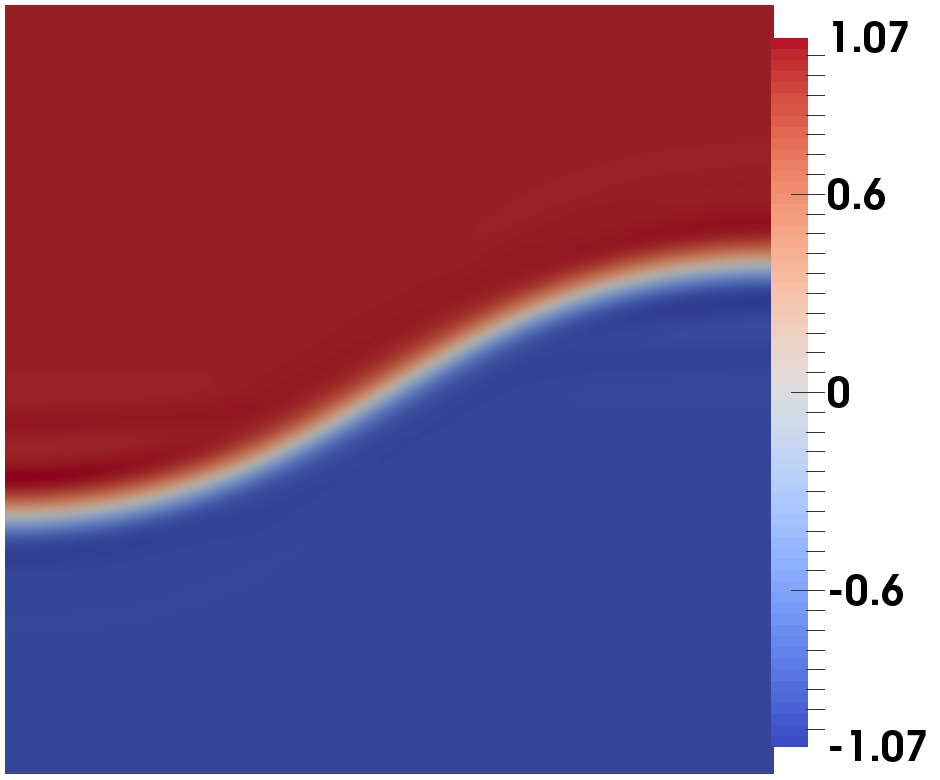} \hspace{-0.2cm}
   \includegraphics[scale=0.1113]{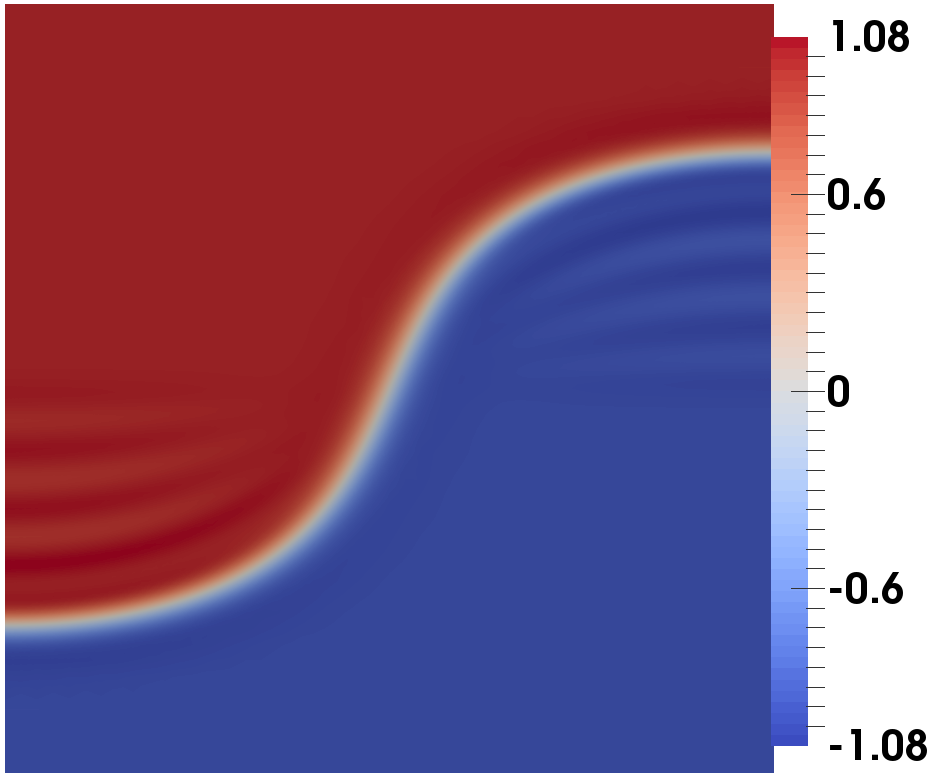}\\
       \includegraphics[scale=0.1113]{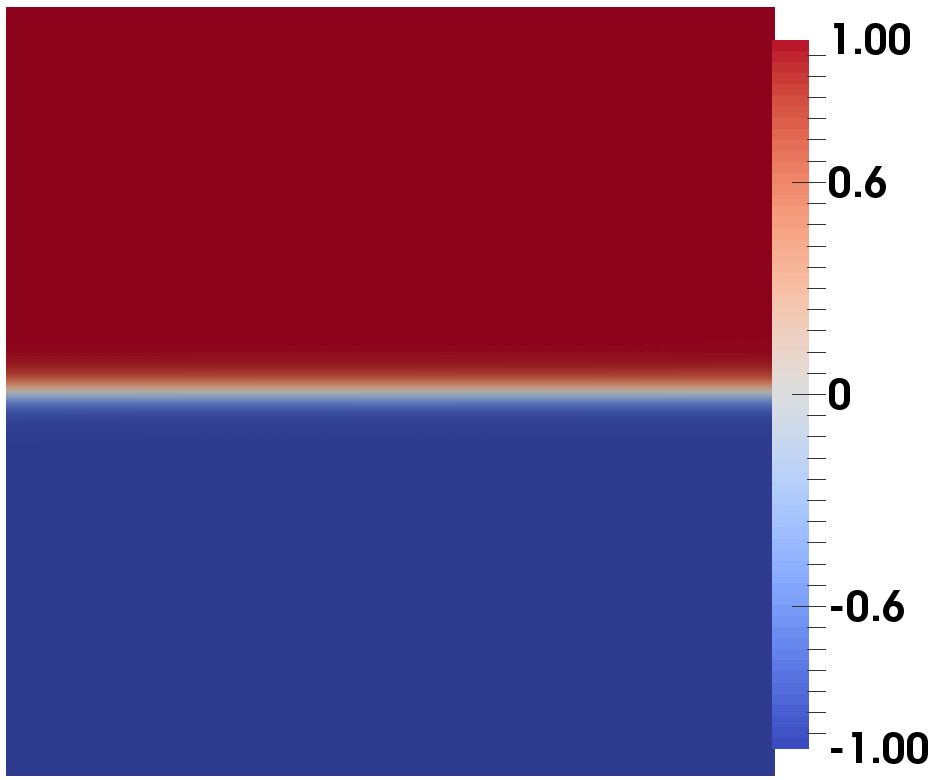} \hspace{-0.2cm}
   \includegraphics[scale=0.1113]{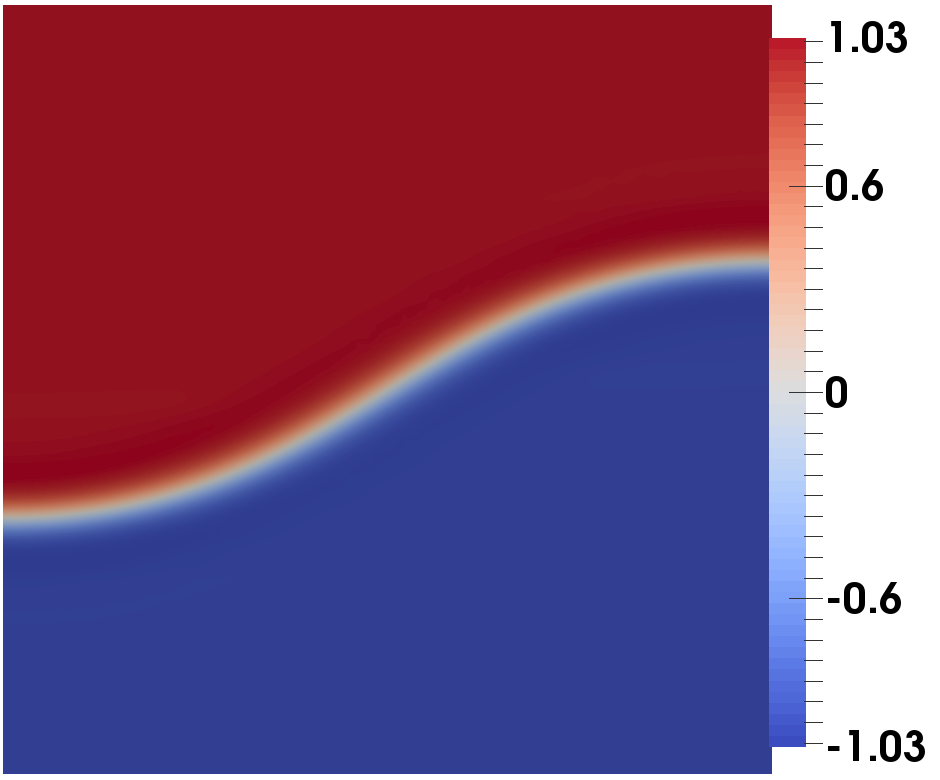} \hspace{-0.2cm}
   \includegraphics[scale=0.1113]{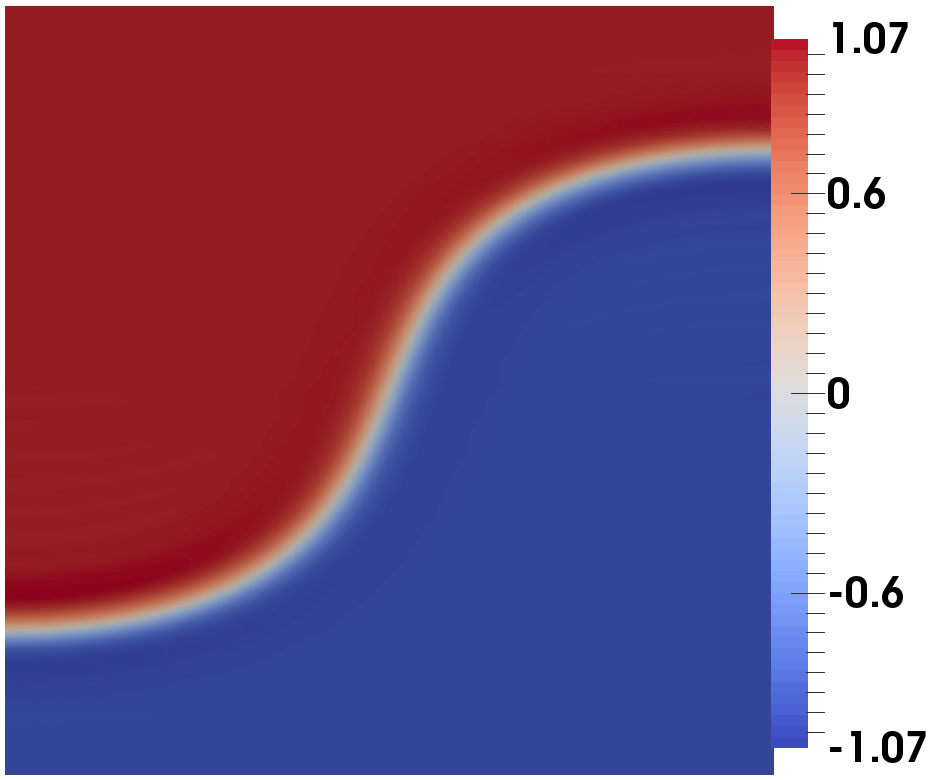}
    \caption{\em Finite element (top) and POD-DEIM optimal solution with $\ell=10$ (middle) and $\ell=20$ (bottom) of the phase field $\varphi$ at $t=t_0, t_{250}, t_{500}$ with adaptive 
    meshes}
  \label{fig:c}
  \end{figure}
  Table 1 (left) summarizes the iteration history for the finite element projected gradient method where we used the stopping criterion 
  $\|\hat{J}'(u^k)\|_{U_h} < 0.01\cdot\|\hat{J}'(u^0)\|_{U_h} + 0.01. $ Table 1 (right) 
  tabulates the POD-ROM optimization. Note that the value of the POD cost functional 
  $\hat{J}_\ell(u^k)$ stagnates 
  due to the POD error. The value of the full-order cost functional at the POD 
  solution is $\hat{J}(\bar{u}_{\text{POD}}) = 7.31 \cdot 10^{-4}$.
  If $\ell=20$ POD modes are used, the relative $L^2(0,T;L^2(\Omega))$-error
  between the finite element and the POD solution for the phase field is err$_\varphi = 7.19 \cdot 10^{-3} $; for POD-DEIM it is err$_\varphi = 7.38 \cdot 10^{-3} $.
  \begin{table}[H]
\centering
 \begin{tabular}{ c | c | c | c  }
  $k$ & $\hat{J}(u^k)$ & $\| \hat{J}'(u^k) \|_{U_h}$  & $s_k$   \\
 \hline
 0  &  $8.61 \cdot 10^{0}$  & $2.85 \cdot 10^0$ & $1.0$  \\
 1 & $6.48 \cdot 10^{-1}$ & $2.32 \cdot 10^0$ & $0.25$ \\ 
 2 & $1.90 \cdot 10^{-2}$ & $4.56 \cdot 10^{-1}$ & $0.25$ \\
 3 & $3.82 \cdot 10^{-3}$ & $1.93 \cdot 10^{-1}$ & $0.25$  \\
 4 & $1.18 \cdot 10^{-3}$ & $8.45 \cdot 10^{-2}$& $0.25$ \\
 5 & $6.80 \cdot 10^{-4}$ & $3.67 \cdot 10^{-2}$ &  \\
  \end{tabular} \hspace{0.5cm}
   \begin{tabular}{ c | c | c | c  }
  $k$ & $\hat{J}_\ell(u^k)$ & $\| \hat{J}_{\ell}'(u^k) \|_{U_h}$  & $s_k$  \\
 \hline
 0  &  $8.77 \cdot 10^{0}$  & $2.81 \cdot 10^0$ & $1.0$  \\
 1 & $7.98 \cdot 10^{-1}$ & $2.41 \cdot 10^0$ & $0.25$ \\ 
 2 & $5.79 \cdot 10^{-2}$ & $3.67 \cdot 10^{-1}$ & $0.25$ \\
 3 & $5.02 \cdot 10^{-2}$ & $1.63 \cdot 10^{-1}$ & $0.25$  \\
 4 & $4.76 \cdot 10^{-2}$ & $7.45 \cdot 10^{-2}$& $0.25$ \\
 5 & $4.76 \cdot 10^{-2}$ & $3.48 \cdot 10^{-2}$ & \\
  \end{tabular}
  \vspace{0.5cm}
  \label{tab:iter_hist}
  \caption{\em Iteration history finite element optimization (left) and POD optimization (right) with $\ell=20$. The Armijo step size is denoted by $s_k$. }
   \end{table}
  \vspace{-0.5cm}
In Table 2 the computational times for the uniform FE, adaptive FE, POD and POD-DEIM optimization are listed. The offline costs for POD when using spatially adapted snapshots are as follows: the interpolation of the snapshots takes 212 seconds, the POD basis computation costs 40 seconds and the computations for DEIM take 30 seconds. In comparison, the use of uniformly discretized snapshots leads to the computational time of 243 seconds for POD basis computation and 193 seconds for the DEIM computations.

   \begin{table}[H]
\centering
\begin{tabular}{ l | c | c | c | c }
                &  \; uniform FE   \;   &  \; adaptive FE \;   & POD & \; POD-DEIM  \\
                \hline
   optimization &  36868 sec       &    5805 sec            & 675 sec  & \hspace{0.05cm} 0.3 sec         \\
  \hline
  $\to$ solve each state eq. & \hspace{0.05cm} 1660 sec  &  \hspace{0.05cm} 348 sec & 42 sec &  0.02 sec\\
  $\to$ solve each adjoint eq. & \hspace{0.21cm} 761 sec   & \hspace{0.05cm} 121 sec & 16 sec & 0.01 sec\\
 \end{tabular} \vspace{0.5cm}
    \label{tab:iter_hist}
  \caption{\em Computational times for the FE, POD and POD-DEIM optimization.}
   \end{table}

   \section{Outlook}
   In future work, we intend to embed the optimization of Cahn-Hilliard in a trust-region framework in order to adapt the POD model accuracy within the optimization. We further want to consider a relaxed double-obstacle free energy which is a smooth approximation of the non-smooth double-obstacle free energy. We expect that more POD modes are needed in this case to get similar accuracy results as in the case of 
   a polynomial free energy. Moreover, we intend to couple the smoothness of the model to the trust-region fidelity.

\section*{Acknowledgments}  
We like to thank Christian Kahle for providing many libraries which we could use for the coding. The first author gratefully acknowledges the financial support by the DFG through the priority program SPP 1962. The third author gratefully acknowledges the financial support by the DFG through the Collaborative Research Center SFB/TRR 181.
%
%
%

\ifx\undefined\bysame
\newcommand{\bysame}{\leavevmode\hbox to3em{\hrulefill}\,}
\fi

\end{document}